\documentclass[12pt]{article}

\usepackage[utf8]{inputenc}
\usepackage[T1]{fontenc}

\usepackage{amsmath,amsfonts,amssymb,amsthm}
\usepackage{multirow}

\usepackage[left=2cm,right=1.5cm,top=1.5cm,bottom=1.5cm]{geometry}

\usepackage[pdfencoding=auto, psdextra, citecolor=blue, urlcolor=blue,
linkcolor=blue, colorlinks=true, bookmarksopen=true]{hyperref}
\usepackage{tablefootnote}

\usepackage[shortlabels]{enumitem}
\setlist[enumerate,1]{label={(\alph*)},noitemsep,topsep=0.5pt}

\usepackage{comment}

\makeatletter
\newcommand\labeltext[2]{#1\def\@currentlabel{#1}\label{#2}}
\makeatother

\newcommand{\R}{\mathbb{R}}
\newcommand{\CC}{\mathbb{C}}
\newcommand{\g}{\mathfrak{g}}
\newcommand{\h}{\mathfrak{h}}
\newcommand{\sll}{\mathfrak{sl}}
\newcommand{\su}{\mathfrak{su}}
\newcommand{\sh}{\mathfrak{sh}}
\newcommand{\se}{\mathfrak{se}}
\newcommand{\SL}{\mathrm{SL}}
\newcommand{\SU}{\mathrm{SU}}
\newcommand{\id}{\mathrm{id}}
\newcommand{\A}{\mathcal{A}}

\DeclareMathOperator{\dist}{dist}
\DeclareMathOperator{\Ker}{Ker}
\DeclareMathOperator{\interior}{int}
\DeclareMathOperator{\ri}{ri}
\DeclareMathOperator{\sspan}{span}
\DeclareMathOperator{\Image}{Im}
\DeclareMathOperator{\Real}{Re}
\DeclareMathOperator{\ad}{ad}
\DeclareMathOperator{\tr}{tr}
\DeclareMathOperator{\arctg}{arctan}

\newtheorem{theorem}{Theorem}

\newtheorem{lemma}{Lemma}
\newtheorem{corollary}{Corollary}

\theoremstyle{definition}
\newtheorem{definition}{Definition}

\newtheorem{remark}{Remark}

\title{Existence of the longest arcs for left-invariant three-dimensional contact sub-Lorentzian structures}
\author{A.\,V.~Podobryaev\\
A.\,K.~Ailamazyan Program Systems Institute of RAS\\
\texttt{alex@alex.botik.ru}}
\date{}

\begin{document}

\maketitle

\begin{abstract}
The problem of finding optimal curves (the longest arcs) for sub-Lorentzian structures is an optimal control problem with an unbounded control set and a concave cost functional. The question of existence of an optimal solution is nontrivial for such problems.
We solve here this question for some left-invariant three-dimensional contact sub-Lorentzian structures, whose classification is known.
We propose sufficient conditions for the existence of the longest arcs for left-invariant (sub-)Lorentzian structures on solvable Lie groups and on the universal cover of the Lie group $\SL_2(\R)$.


\end{abstract}

\section{\label{sec-intoduction}Introduction}

Let us recall the definitions of a sub-Lorentzian structure on a smooth manifold, admissible curves and their lengths, sub-Lorentzian distance, and the longest arcs.

\begin{definition}
\label{def-sub-Lorentzian}
\emph{A sub-Lorentzian structure} on a smooth manifold $M$ is a smooth completely non-integrable distribution $\Delta$, endowed with a smooth Lorentzian metric $q$.
Such a structure is called \emph{contact} if the distribution $\Delta$ is contact, i.e., it is the kernel of a 1-form $\omega$ such that
the form $\Lambda^m (d\omega) \wedge \omega$ does not vanish, where $\dim{M} = 2m+1$.

Define the cone of \emph{non-spacelike} vectors $C_x = \{v \in \Delta_x \subset T_xM \, | \, q_x(v,v) \leqslant 0\}$ in the tangent space at each point $x \in M$,.
Assume that a \emph{time orientation} is given, i.e., a \emph{timelike} vector field $V$ is fixed, where $V_x \in \Delta_x$ and $q_x(V_x, V_x) < 0$ for any $x \in M$.
The cone $C_x^+ = \{v \in C_x \, | \, q_x(V_x, v) \leqslant 0 \} \subset T_xM$ is called \emph{the cone of admissible velocities}.
A curve $x : [0,t_1] \rightarrow M$ is called \emph{an admissible curve} if $\dot{x}(t) \in C^+_{x(t)}$ for almost all $t \in [0,t_1]$.
\emph{The sub-Lorentzian length} of an admissible curve is defined as
$$
\int\limits_0^{t_1}{\sqrt{|q_{x(t)}(\dot{x}(t), \dot{x}(t))|} \, dt}.
$$
\emph{The sub-Lorentzian distance} $\dist_{sL}(x_0, x_1)$ from point $x_0$ to point $x_1$ is the supremum of the lengths of admissible curves from $x_0$ to $x_1$,
if it exists. If there exists a curve on which this supremum is attained, it is called \emph{the longest arc} from $x_0$ to $x_1$.
\end{definition}

A sub-Lorentzian structure is called \emph{left-invariant} if it is defined on a Lie group and is invariant under left shifts. We can assume that the point $x_0$ is the group identity in this case.

Sub-Lorentzian structures have been actively studied recently. Both general questions~\cite{grochowski,lokutsievskiy-podobryaev} and the structure of specific sub-Lorentzian manifolds are investigated (geodesics, attainable sets, the longest arcs are described)~\cite{chang-markina-vasilev,grong-vasilev,berestovskii-gedel,sachkov-sl-martinet}. Moreover, sub-Lorentzian structures arise in control theory. They correspond to optimal control problems on manifolds with a drift and control lying in a disk, with a quadratic cost functional.

The question of existence of the longest arcs is nontrivial. Indeed, from the viewpoint of control theory, the problem of finding the longest arcs is an optimal control problem where the control set is unbounded and the cost functional is concave, so the standard Filippov theorem~\cite[Ch.~2, \S\,7, vol.~3]{filippov} is not applicable in this case.

The present paper shows the existence of the longest arcs for certain left-invariant three-dimensional contact sub-Lorentzian structures.
Such sub-Lorentzian structures are classified up to local isometry preserving time and space orientations in the work of
M.~Grochowski, A.~Medvedev and B.~Warhurst~\cite{classification}. They are specified by a set of invariants. We need to introduce some notations to define these invariants.

We may assume that the left-invariant sub-Lorentzian structure is defined on a Lie group $G$ with the help of elements of the Lie algebra $X_1, X_2 \in \g = T_{\id}G$ such that
$$
X_1, X_2 \in \Delta_{\id}, \qquad q_{\id}(X_1, X_1) = -1, \qquad q_{\id}(X_1, X_2) = 0, \qquad q_{\id}(X_2, X_2) = 1,
$$
where $\id$ denotes the identity element of the group $G$.
Let the contact form $\omega$ be such that $d\omega(X_1, X_2) = \omega([X_1, X_2]) = 1$.
Choose a vector $X_3 \in \g$ such that its left shifts form the Reeb vector field for the form $\omega$, i.e., $d\omega(X_3, \,\cdot\,) = 0$, $\omega(X_3) = -1$.
Then the commutation table is determined by the structure constants~\cite[Sec.~4.1]{classification}
\begin{equation}
\label{eq-commutators}
[X_1,X_3] = c X_1 + c_{13}^2 X_2, \qquad
[X_2,X_3] = c_{23}^1 X_1 - c X_2, \qquad
[X_1,X_2] = c_{12}^1 X_1 + c_{12}^2 X_2 + X_3.
\end{equation}
The invariants $h$, $\kappa$, and $\tau$ are given by the formulas:
$$
h = \left(
\begin{array}{cc}
c & \frac{c_{23}^1 - c_{13}^2}{2} \\
\frac{c_{13}^2 - c_{23}^1}{2} & c \\
\end{array}
\right), \qquad
\kappa = -(c_{12}^1)^2 + (c_{12}^2)^2 + \frac{c_{23}^1 + c_{13}^2}{2}, \qquad
\tau = c_{12}^1.
$$
The classification of the considered sub-Lorentzian structures is given in Table~\ref{tb-classification}~\cite[Th.~1]{classification}.
This table uses notation for three-dimensional Lie algebras from the Bianchi classification~\cite[Th.~3]{classification}, see also~\cite{alglie3d}.
The list of these Lie algebras with commutation relations is given in Table~\ref{tb-alglie3d}.
Here $\h_3$ denotes the Lie algebra of the Heisenberg group of upper unitriangular $3 \times 3$ matrices,
$\sh_2$ --- the Lie algebra of the group of isometries of Minkowski space $\R^{1,1}$,
$\se_2$ --- the Lie algebra of the group of isometries of Euclidean space $\R^2$,
and $A^+(\R)$ --- the Lie algebra of the group of affine transformations of the line $\R$.

For $h \neq 0$, the sub-Lorentzian structure is defined uniquely~\cite[Crl.~1, Prop.~4]{classification}.
If $h = 0$, then for $\kappa \neq 0$, the sub-Lorentzian structures on the Lie groups $\SL_2(\R)$ and $A^+(\R) \times \R$ are locally isometric~\cite[Th.~4]{classification},
they correspond to cases~\ref{2} and \ref{2suppl} of Table~\ref{tb-classification}, respectively.
For $\det{h} \leqslant 0$, the sub-Lorentzian structure is completely determined by the invariants $h, \kappa$, and $\tau$ (if the latter is not uniquely determined by the first two).
For $\det{h} > 0$, each set of values $h, \kappa$ corresponds to three sub-Lorentzian structures.

Let us state the main theorem on the existence of the longest arcs for left-invariant three-dimensional contact sub-Lorentzian structures.
Its assertion~(1) corresponds to the case of solvable groups and assertions~(2)--(3) correspond to the cases of semisimple groups.

\begin{theorem}
\label{th-exist3d}
\emph{(1)} For left-invariant sub-Lorentzian structures~\emph{\ref{1}}, \emph{\ref{11}--\ref{12}} \emph{(}for $\chi = \kappa$\emph{)} and \emph{\ref{13}--\ref{15}}
\emph{(}see Table~\emph{\ref{tb-classification}}\emph{)} on simply connected solvable Lie groups, for any two points $x_0$ and $x_1$ there exists the longest arc from $x_0$ to $x_1$ if and only if point $x_1$ is attainable from point $x_0$.\\
\emph{(2)} For left-invariant sub-Lorentzian structures on the simply connected covering $\widetilde{\SL}_2(\R)$ of the Lie group $\SL_2(\R)$
in case~\emph{\ref{10}} for $\kappa < \chi < 0$, for any two points $x_0$ and $x_1$ there exists the longest arc from $x_0$ to $x_1$ if and only if point $x_1$ is attainable from point $x_0$.\\
\emph{(3)} For left-invariant sub-Lorentzian structures on the Lie group $\SU_2$ \emph{(}case~\emph{\ref{9}}\emph{)} for any points $x_0$ and $x_1$ connectable by an admissible curve,
$\dist_{sL}(x_0,x_1) = +\infty$.
\end{theorem}

\begin{table}[ht]
\begin{footnotesize}
\begin{tabular}{|l|l|l|l|l|l|}
\hline
                    & $h$                                                                                                                                                     & $\kappa$                                                    & $\tau$            & Lie algebra                                                                                                                & Existence        \\ \hline\hline
\labeltext{1}{1}                   & \multirow{3}{*}{$\left(\begin{array}{cc}0 & 0 \\ 0 & 0 \\ \end{array} \right)$}                                                                         & $\kappa = 0$                                                &                   & $L(3,1) = \h_3$                                                                                                           & Yes                   \\ \cline{1-1} \cline{3-6}
\labeltext{2}{2}                   &                                                                                                                                                         & $\kappa \in \R^*$                                           &                   & $L(3,5) = \sll_2(\R)$                                                                                                         &                  \\ \cline{1-1} \cline{3-6}
\labeltext{$2^*$}{2suppl}                   &                                                                                                                                                         & $\kappa \in \R^*$                                           &                   & $L(3,-1) = A^+(\R) \times \R$                                                                                                         & Yes                \\ \hline
\labeltext{3}{3}                   & \multirow{4}{*}{$\left(\begin{array}{rr}1 & -1 \\ 1 & -1 \\ \end{array}\right), \quad \left(\begin{array}{rr} -1 & -1 \\ 1 & 1 \\ \end{array} \right)$} & $\kappa = 0$                                                & $\tau = 2$        & $L(3,3)$                                                                                                                  &                      \\ \cline{1-1} \cline{3-6}
\labeltext{4}{4}                   &                                                                                                                                                         & $\kappa = 0$                                                & $|\tau| > 2$      & $L\left(3,2,\frac{-\tau - \sqrt{\tau^2 - 4}}{-\tau + \sqrt{\tau^2 - 4}}\right)$                                           &                      \\ \cline{1-1} \cline{3-6}
\labeltext{5}{5}                   &                                                                                                                                                         & $\kappa = 0$                                                & $|\tau| < 2$      & $L\left(3,4,\frac{|\tau|}{\sqrt{4 - \tau^2}}\right)$                                                                      &                      \\ \cline{1-1} \cline{3-6}
\labeltext{6}{6}                   &                                                                                                                                                         & $\kappa \in \R^*$                                           &                   & $L(3,5) = \sll_2(\R)$                                                                                                         &                   \\ \hline
\labeltext{7}{7}                   & \multirow{2}{*}{$\left(\begin{array}{rr} 1 & 1 \\ -1 & -1 \\ \end{array}\right), \quad \left(\begin{array}{rr}-1 & 1 \\ -1 & 1 \\ \end{array}\right)$}  & $\kappa = 0$                                                & $\tau \in \R$     & $L\left(3,2,\frac{\tau - \sqrt{\tau^2 + 4}}{\tau + \sqrt{\tau^2 + 4}}\right)$                                             &                      \\ \cline{1-1} \cline{3-6}
\labeltext{8}{8}                   &                                                                                                                                                         & $\kappa \in \R^*$                                           &                   & $L(3,5) = \sll_2(\R)$                                                                                                         &                   \\ \hline
\labeltext{9}{9}                   & \multirow{5}{*}{$\left(\begin{array}{rr} 0 & -\chi \\ \chi & 0 \\ \end{array} \right), \quad \chi \neq 0$}                                              & $|\kappa| < -\chi$                                          &                   & $L(3,6) = \su_2$                                                                                                          & No\tablefootnote{The Lorentzian distance to any attainable point is $+\infty$ for any starting point.}                  \\ \cline{1-1} \cline{3-6}
\labeltext{10}{10} &                                                                                                                                                         & $|\kappa| > -\chi, \quad \chi \neq \pm \kappa$ &  & $L(3,5) = \sll_2(\R)$                                                                                        & Yes, if $\kappa < \chi < 0$ \\ \cline{1-1} \cline{3-6}
\labeltext{11}{11}                  &                                                                                                                                                         & $\chi = \pm \kappa > 0$                                     &                   & $L(3,2,-1) = \sh_2$                                                                                                       & Yes, if $\chi = \kappa$                   \\ \cline{1-1} \cline{3-6}
\labeltext{12}{12}                  &                                                                                                                                                         & $\chi = \pm \kappa < 0$                                     &                   & $L(3,4,0) = \se_2$                                                                                                        & Yes, if $\chi = \kappa$                   \\ \hline
\labeltext{13}{13}                  & \multirow{3}{*}{$\left(\begin{array}{rr} 0 & -\chi \\ \chi & 0 \\ \end{array}\right), \quad \chi \neq 0$}                                               & $\kappa = -7\chi$                                           &                   & $L(3,3)$                                                                                                                  & Yes                   \\ \cline{1-1} \cline{3-6}
\labeltext{14}{14}                  &                                                                                                                                                         & $\kappa > -7\chi$                                           &                   & $L\left(3,2,\frac{\sqrt{|\chi - \kappa|} - \sqrt{7\chi + \kappa}}{\sqrt{|\chi - \kappa|} + \sqrt{7\chi + \kappa}}\right)$ & Yes                   \\ \cline{1-1} \cline{3-6}
\labeltext{15}{15}                  &                                                                                                                                                         & $\kappa < -7\chi$                                           &                   & $L\left(3,4,\frac{\sqrt{|\kappa - \chi|}}{\sqrt{-7\chi - \kappa}}\right)$                                                 & Yes                   \\ \hline
\labeltext{16}{16}                  & \multirow{3}{*}{$\left(\begin{array}{rr} 0 & -\chi \\ \chi & 0 \\ \end{array}\right), \quad \chi \neq 0$}                                               & $\kappa = 7\chi$                                            &                   & $L(3,3)$                                                                                                                  &                      \\ \cline{1-1} \cline{3-6}
\labeltext{17}{17}                  &                                                                                                                                                         & $\kappa < 7\chi$                                            &                   & $L\left(3,2,\frac{\sqrt{|\chi + \kappa|} - \sqrt{7\chi - \kappa}}{\sqrt{|\chi + \kappa|} + \sqrt{7\chi - \kappa}}\right)$ &                      \\ \cline{1-1} \cline{3-6}
\labeltext{18}{18}                  &                                                                                                                                                         & $\kappa > 7\chi$                                            &                   & $L\left(3,4,\frac{\sqrt{|\kappa + \chi|}}{\sqrt{-7\chi + \kappa}}\right)$                                                 &                      \\ \hline
\labeltext{19}{19}                  & $\left(\begin{array}{rr} \chi & 0 \\ 0 & -\chi \\ \end{array} \right), \quad \chi \neq 0$                                                               & $\kappa \in \R$                                             &                   & $L(3,5) = \sll_2(\R)$                                                                                                         &                   \\ \hline
\end{tabular}
\end{footnotesize}
\caption{\label{tb-classification}The Grochowski-Medvedev-Warhurst classification of left-invariant three-dimensional contact sub-Lorentzian structures. It is supplemented with the last column indicating the existence of the longest arcs on the attainable set on the corresponding simply connected Lie group (Theorem~\ref{th-exist3d}).}
\end{table}

\begin{table}[h!]
\centering
\begin{small}
\begin{tabular}{|c|c|}
\hline
Lie algebra & Commutation relations \\
\hline\hline
$L(3, 0) = \R^3$ & $[E_1, E_2] = [E_2, E_3] = [E_3, E_1] = 0$ \\
\hline
$L(3,1) = \h_3$ & $[E_1, E_2] = E_3$, $[E_1, E_3] = [E_2, E_3] = 0$ \\
\hline
$L(3, -1) = A^+(\R) \oplus \R$ & $[E_1, E_2] = E_1$, $[E_1, E_3] = [E_2, E_3] = 0$ \\
\hline
$L(3,2,\eta)$, $0 < |\eta| \leqslant 1$ & $[E_1, E_2] = 0$, $[E_1, E_3] = E_1$, $[E_2, E_3] = \eta E_2$ \\
\hline
$L(3,4,\eta)$, $\eta \geqslant 0$ & $[E_1, E_2] = 0$, $[E_1, E_3] = \eta E_1 - E_2$, $[E_2, E_3] = E_1 + \eta E_2$ \\
\hline
$L(3,3)$ & $[E_1, E_2] = 0$, $[E_1, E_3] = E_1$, $[E_2, E_3] = E_1 + E_2$ \\
\hline
$L(3,5) = \sll_2(\R)$ & $[E_1, E_2] = E_1$, $[E_1, E_3] = -2E_2$, $[E_2, E_3] = E_3$ \\
\hline
$L(3,6) = \su_2$ & $[E_1, E_2] = E_3$, $[E_1, E_3] = -E_2$, $[E_2, E_3] = E_1$ \\
\hline
\end{tabular}
\end{small}
\caption{\label{tb-alglie3d}Three-dimensional Lie algebras with generators $E_1, E_2, E_3$.}
\end{table}

Theorem~\ref{th-exist3d} can be extended to the case of generalized sub-Lorentzian structures,
which have attracted interest in recent years~\cite{lokutsievskiy-podobryaev,podobryaev,ladeishchikov}.
By analogy with sub-Finsler structures~\cite{berestovslii-intistic}, defined by a norm on a distribution,
generalized sub-Lorentzian structures are defined by an anti-norm on the cone of admissible velocities,
i.e., a nonnegative concave homogeneous function of degree 1. The generalized sub-Lorentzian length functional is defined as the integral of the anti-norm of the velocity of an admissible curve.
See Definition~\ref{def-general-subLorentz} in Section~\ref{sec-maintheorems} below for more details. Let us state the corresponding assertion here.

\begin{theorem}
\label{th-exist-general}
The existence or non-existence of the longest arcs in the cases listed in Theorem~\emph{\ref{th-exist3d}} also holds for left-invariant generalized sub-Lorentzian structures given by the cone
$C^+_{\id} = \{(x_1,x_2,x_3) \in \g \, | \, -x_1^2 + x_2^2 \leqslant 0, \, x_1 > 0, \, x_3 = 0\}$ and an arbitrary upper semicontinuous anti-norm on this cone,
where $x_1,x_2,x_3$ are coordinates in the Lie algebra $\g$ associated with the basis $X_1,X_2,X_3$.
\end{theorem}

Theorem~\ref{th-exist3d} is a particular case of Theorem~\ref{th-exist-general}. So, we will prove the latter.
The assertions of Theorem~\ref{th-exist-general} corresponding to assertions~(1)--(3) of Theorem~\ref{th-exist3d} we call also assertions~(1)--(3).

\begin{remark}
\label{rem-su2}
Assertion~(3) of Theorems~\ref{th-exist3d}, \ref{th-exist-general} follows immediately from the fact that on the Lie group $\SU_2$ there exist closed timelike trajectories passing through any point. For example, the circle $S^1 = \{\exp{(tX_1)} \, | \, t \in \R_+\}$ passing through the group identity can be translated to any given point via a left shift.
Including sequential multiple traversals of a closed timelike curve into an admissible curve leads to unbounded growth of the length of the admissible curve.
\end{remark}

Thus, this paper is dedicated to the proof of assertions~(1)--(2) of Theorems~\ref{th-exist3d}, \ref{th-exist-general}.
Section~\ref{sec-maintheorems} provides some preliminary information on generalized (sub-)Lorentzian structures. There we state a sufficient condition for the existence of the longest arc from the work~\cite{lokutsievskiy-podobryaev} in Theorem~\ref{th-exist}. Then, we specify this sufficient condition for left-invariant (sub-)Lorentzian structures on Lie groups (Corollary~\ref{crl-exist}).
This corollary states that if the cone of admissible velocities intersects the commutator of the Lie algebra only at zero, then the longest arcs leading to attainable points exist.
In Section~\ref{sec-solv}, we apply this condition to left-invariant sub-Lorentzian structures on three-dimensional solvable Lie groups, i.e., we prove assertions~(1) of Theorems~\ref{th-exist3d}, \ref{th-exist-general}.
Finally, in Section~\ref{sec-sl2}, we prove a sufficient condition for the existence of the longest arcs for sub-Lorentzian structures on the group $\widetilde{\SL}_2(\R)$ (Theorem~\ref{th-sl2}), then we apply it to prove assertions~(2) of the main Theorems~\ref{th-exist3d}, \ref{th-exist-general}.

In the cases of unfilled cells in the last column of Table~\ref{tb-classification}, the existence conditions for the longest arcs proposed in this paper (Corollary~\ref{crl-exist} and Theorem~\ref{th-sl2}) are not applicable. These cases require a separate investigation.

\section{\label{sec-maintheorems}Preliminaries}

This section gives a condition for the existence of the longest arcs for generalized left-invariant (sub-) Lorentzian structures.
Below we state the problem of finding the longest arcs as an optimal control problem.

\begin{definition}
\label{def-general-subLorentz}
Consider a finite-dimensional smooth real manifold $M$.
Suppose that in each of its tangent spaces, \emph{an acute cone} $C^+_x \subset T_xM$ is given (i.e., it does not contain any nonzero subspaces) and
an associated upper semicontinuous \emph{anti-norm} $\nu_x : T_xM \rightarrow \R \sqcup \{-\infty\}$, i.e., the following conditions hold:
$$
\begin{array}{lcl}
\nu_x|_{C^+_x} \geqslant 0, & \qquad & \forall\, v \in T_xM \ \, \forall\, \lambda > 0 \qquad \nu_x(\lambda v) = \lambda\nu_x(v),\\
\nu_x|_{T_xM \setminus C^+_x} = -\infty, & \qquad & \forall\, v,w \in T_xM \qquad \nu_x(v+w) \geqslant \nu_x(v) + \nu_x(w).\\
\end{array}
$$
It can be shown (for details see~\cite[Def.~3]{lokutsievskiy-podobryaev}) that in this case the cone $C^+_x$ is convex and closed.
Moreover, if $\nu_x|_{C^+_x} \not\equiv 0$, then $\nu_x|_{\ri{C^+_x}} > 0$, where $\ri{C^+_x}$ is the relative interior of the cone $C^+_x$.
The collection $\{(C^+_x, \nu_x) \, | \, x \in M\}$ is called \emph{a generalized sub-Lorentzian structure} on the manifold $M$.
\end{definition}

\begin{definition}
\label{def-regular}
A generalized sub-Lorentzian structure is called \emph{regular} if the set $C^+ = \bigcup\limits_{x \in M}{C^+_x} \subset TM$ is closed, and
the function $\nu : TM \rightarrow \R \sqcup \{-\infty\}$, acting on the fiber $T_xM$ as $\nu_x$, is upper semicontinuous.
\end{definition}

We need to find a Lipschitz curve $x : [0,t_1] \rightarrow M$ such that
\begin{equation}
\label{eq-control-problem}
x(0) = x_0, \qquad x(t_1) = x_1, \qquad \dot{x}(t) \in C^+_{x(t)} \setminus 0 \text{ for a.e. } t \in [0,t_1], \qquad \int\limits_0^{t_1}{\nu(\dot{x}(t))\, dt} \rightarrow \max,
\end{equation}
where $x_0, x_1 \in M$ and $t_1 > 0$ are fixed.

If a nondegenerate bilinear form $q$ with signature $(1, \dim{M}-1)$ is given on the manifold $M$, then by fixing a timelike vector field $V$,
we can define the cone $C^+_x = \{v \in T_xM \, | \, q_x(v,v) \leqslant 0, \, q_x(v, V_x) \leqslant 0\}$ and the anti-norm $\nu_x|_{C^+_x}(v) = \sqrt{|q_x(v, v)|}$.
In this case, problem~\eqref{eq-control-problem} is the classical left-invariant Lorentzian problem~\cite{beem-ehrlich-easley}.

The statement of problem~\eqref{eq-control-problem} also includes the sub-Lorentzian case, i.e, the cone $C^+_x$ may have empty interior
(be contained in a proper subspace $\Delta_x$ of the tangent space $T_xM$).

The Filippov theorem~\cite[Ch.~2, \S\,7, vol.~3]{filippov} is not directly applicable (admissible velocities are unbounded, and the integrand of the cost functional is concave) to solve the question of existence of an optimal solution of problem~\eqref{eq-control-problem}.

In recent work~\cite{lokutsievskiy-podobryaev}, a sufficient condition for the existence of the longest arcs for regular generalized (sub-)Lorentzian structures on smooth manifolds is proposed.

\begin{theorem}[\cite{lokutsievskiy-podobryaev}, Th.~1]
\label{th-exist}
Consider a complete Riemannian manifold $M$ endowed with a regular generalized sub-Lorentzian structure. Assume that $H^1(M) = 0$. Fix a point $x_0 \in M$.
If there exists a 1-form $\tau$ such that\\
\emph{(1)} $\tau_x|_{C^+_x \setminus 0} > 0$ for any $x \in M$,\\
\emph{(2)} $\frac{|v|}{\tau_x(v)} \leqslant 1 + \dist_R(x_0, x)$ for any $x \in M$ and $v \in T_xM \setminus 0$, where $\dist_R$ is the Riemannian distance, and $|\cdot|$ is the norm in the tangent space corresponding to the Riemannian structure,\\
\emph{(3)} $d\tau = 0$,\\
then a solution of problem~\emph{\eqref{eq-control-problem}}, i.e., the longest arc from the point $x_0$ to a point $x_1$ exists if and only if the point $x_1$ lies in the attainable set from the point $x_0$.
\end{theorem}

\begin{remark}
\label{rem-time-orientation}
In the case of a cone defined by a quadratic form $C_x = \{v \in T_xM \, | \, q_x(v,v) \leqslant 0\}$, the 1-form $\tau$ selects a closed convex future cone $C^+_x$,
whose relative interior lies on one side of the hyperplane $\Ker{\tau_x}$.
Therefore, the 1-form $\tau$ is called \emph{a time orientation form}.
\end{remark}

The conditions of Theorem~\ref{th-exist} simplify significantly for left-invariant sub-Lorentzian structures and left-invariant time orientation forms.

\begin{corollary}
\label{crl-exist}
Consider a left-invariant generalized sub-Lorentzian structure on a Lie group $G$.
Assume that $H^1(G) = 0$ and $C^+_{\id} \cap [\g, \g] = 0$, where $\g$ is the corresponding Lie algebra.
Then the longest arc from the point $\id$ to a point $x_1 \in G$ exists if and only if
the point $x_1$ lies in the attainable set from the point $\id$.
\end{corollary}

\begin{proof}
Note that on a Lie group there exists a left-invariant Riemannian metric, which is complete~\cite[Lemma~5.4]{beem-ehrlich-easley}.
Moreover, for left-invariant generalized sub-Lorentzian structures, the regularity condition holds automatically.

It remains to construct a suitable time orientation 1-form $\tau$ to apply Theorem~\ref{th-exist}.
We take a left-invariant 1-form such that $\Ker{\tau_{\id}} \supset [\g,\g]$ and $\tau_{\id}|_{C^+_{\id} \setminus 0} > 0$ as a time orientation form,
i.e., as a form satisfying condition~(1) of Theorem~\ref{th-exist}. This is possible due to $C^+_{\id} \cap [\g, \g] = 0$.

For a left-invariant 1-form $\tau$ and a left-invariant Riemannian metric, condition~(2) is satisfied.

It remains to note that for left-invariant vector fields $V, W$, we have $d\tau(V, W) = \tau_{\id}([V_{\id}, W_{\id}]) = 0$, since $\Ker{\tau_{\id}} \supset [\g,\g]$.
Hence, the form $\tau$ is closed and condition~(3) of Theorem~\ref{th-exist} is satisfied as well.
\end{proof}

\begin{remark}
\label{rem-small-big}
Clearly, if the longest arcs exist for a left-invariant generalized sub-Lorentzian structure given by a cone $C^+ \subset \g$, then
the longest arcs also exist for a structure given by a cone $c^+ \subset C^+$.
Moreover, under the conditions of Corollary~\ref{crl-exist}, the existence of the longest arcs does not depend on the choice of an anti-norm.
\end{remark}

\begin{remark}
\label{rem-solv-semisimple}
Note that a left-invariant closed time orientation 1-form can exist for left-invariant sub-Lorentzian structures on solvable Lie groups.
Such structures are listed below in Section~\ref{sec-solv}. But a left-invariant closed time 1-form cannot exist for semisimple Lie groups, since in this case the Lie algebra coincides with its commutator $\g = [\g,\g]$. Therefore, in Section~\ref{sec-sl2}, we construct a non-left-invariant closed time orientation 1-form
on the universal cover of the group $\SL_2(\R)$ under certain constraints for which we specifically verify condition~(2) of Theorem~\ref{th-exist}.
\end{remark}

\begin{remark}
\label{rem-simpleconditions}
The conditions of Theorem~\ref{th-exist} and especially Corollary~\ref{crl-exist} are simpler than the classical condition for existence of the longest arcs
(the condition of global hyperbolicity~\cite{beem-ehrlich-easley}, generalized also to the sub-Lorentzian case in papers~\cite{grochowski,sachkov}), that
requires to check compactness of the set $\A_{x_0}^+ \cap \A_{x_1}^-$, where $\A_{x_0}^+$ and $\A_{x_1}^-$ are the attainable sets from points $x_0$, $x_1$ under forward and backward motion, respectively. We illustrate this with the present work.
\end{remark}

\section{\label{sec-solv}Sub-Lorentzian structures on three-dimensional solvable Lie groups}

We need some notations and two technical lemmas.

For a cone $C \subset V$ in a vector space $V$, we denote by $\interior{C}$ its interior,
by $C^{\vee} = \{p \in V^* \, | \, p|_C \geqslant 0\}$ the dual cone,
and by $U^{0} \subset V^*$ the annihilator of a subspace $U \subset V$.

\begin{lemma}
\label{lem-ri}
For an acute cone $C$, we have $\interior{(C^{\vee})} = \{p \in V^* \, | \, p|_{C \setminus 0} > 0\} \neq \varnothing$.
\end{lemma}

\begin{proof}
First note that if the cone $C$ is acute, then $\interior{(C^{\vee})} \neq \varnothing$.
Indeed, otherwise the dual cone $C^{\vee}$ lie in some proper subspace,
whose annihilator is contained in the cone $C$, in contradiction with its acuteness.

Let $p \in \interior{(C^{\vee})}$. Let us show that $p|_{C \setminus 0} > 0$.
Assume by contradiction, that there exists $u \in C \setminus 0$ such that $p(u) = 0$.
Since $p \in \interior{(C^{\vee})}$, for any $r \in V^*$ there exists $\varepsilon > 0$ such that $p \pm \varepsilon r \in C^{\vee}$, i.e.,
$p(u) \pm \varepsilon r(u) = \pm \varepsilon r(u) \geqslant 0$. Hence, $r(u) = 0$ for any $r \in V^*$. It follows that $u = 0$, we get a contradiction.

Conversely, let $p|_{C \setminus 0} > 0$. Suppose by contradiction that $p \notin \interior{(C^{\vee})}$, i.e.,
for $\varepsilon_n = \frac{1}{n} > 0$ there exists $q_n \in V^*$ such that $|p - q_n| < \varepsilon_n$ and $q_n \notin C^{\vee}$,
where $|\,\cdot\,|$ denotes the Euclidean norm in the space $V^*$.
This means that there exists $u_n \in C \setminus 0$ such that $q_n(u_n) < 0$.
We can assume that $|u_n| = 1$. Passing to a subsequence converging to an element $u \in C \setminus 0$,
we obtain that $q_n(u_n)$ converges to $p(u)$, so $p(u) \leqslant 0$. We get a contradiction.
\end{proof}

\begin{lemma}
\label{lem-dual}
Consider a closed convex acute cone $C$ in a vector space $V$ and a subspace $U \subset V$.
Then the following conditions are equivalent\emph{:}\\
\emph{(1)} $C \cap U = 0$,\\
\emph{(2)} $\interior{(C^{\vee})} \cap U^{0} \neq \varnothing$.
\end{lemma}

\begin{proof}
(1) $\Rightarrow$ (2).
Consider the projection $\pi : V \rightarrow V/U$.
It follows from the linearity of the map $\pi$ that $\pi(C)$ is a closed convex cone.
Moreover, this cone is acute. Indeed, otherwise there exists $v+U \in \pi(C)$ such that $-v+U \in \pi(C)$.
Then in the preimages of these elements under the map $\pi$, there exist elements $c_1, c_2 \in C$ (respectively) such that
$c_1 + c_2 \in U$. If $c_1+c_2 = 0$, then the cone $C$ contains the line $\sspan{\{c_1\}}$ and is not acute.
If $c = c_1+c_2 \neq 0$, then $c \in C \cap U \neq 0$, we get a contradiction.

Further, due to the acuteness of the cone $\pi(C)$, by Lemma~\ref{lem-ri} there exists $p \in (V/U)^* \cong U^{0}$ such that $p|_{\pi(C) \setminus 0} > 0$.
Then $p \in \interior{(C^{\vee})} \cap U^{0} \neq \varnothing$.

(2) $\Rightarrow$ (1).
The condition $\interior{(C^{\vee})} \cap U^{0} \neq \varnothing$ is equivalent to the existence of $p \in \interior{(C^{\vee})}$ such that for any $u \in U$, we have $p(u) = 0$.
In other words, $u$ vanishes on the interior of the cone $C^{\vee}$.
Then for a nonzero $u$, two cases are possible: either $u|_{C^{\vee}} \equiv 0$, or $u$ takes values of different signs on the cone $C^{\vee}$.
The first case is impossible, because then $\sspan{\{u\}} \subset C^{\vee\vee} = C$, which contradicts the acuteness of the cone $C$.
Hence, the second case holds and $u \notin C^{\vee\vee} = C$.
Therefore, $C \cap U = 0$.
\end{proof}

\begin{lemma}
\label{lem-anh}
A left-invariant closed time orientation 1-form exists if and only if $[\g,\g]^0 \cap \interior{(C^+_{\id})^{\vee}} \neq \varnothing$.
\end{lemma}

\begin{proof}
Immediately follows from Lemma~\ref{lem-dual} and Corollary~\ref{crl-exist}.
\end{proof}

\begin{proof}[Proof of assertion~\emph{(1)} of Theorems~\emph{\ref{th-exist3d}, \ref{th-exist-general}}]
We apply Corollary~\ref{crl-exist} to sub-Lorentzian structures from Table~\ref{tb-classification} in the case of solvable Lie algebras.
Let us describe the general scheme. The invariants $h, \kappa, \tau$ determine the table of commutators~\eqref{eq-commutators} for the generators $X_1, X_2, X_3$ of the Lie algebra.
The annihilator of the commutator subalgebra $[\g,\g]^0$ is determined by the kernel of the matrix of structure constants (see commutator relations~\eqref{eq-commutators})
$$
A = \left(
\begin{array}{ccc}
c & c_{13}^2 & 0 \\
c_{23}^1 & -c & 0 \\
c_{12}^1 & c_{12}^2 & 1 \\
\end{array}
\right).
$$
In our case, the interior of the dual cone $\interior{(C^+_{\id})^{\vee}}$ is given by the conditions $- x_1^2 + x_2^2 < 0$ and $x_1 > 0$, where
$x_1, x_2, x_3$ are coordinates associated with the basis $X_1, X_2, X_3 \in \g$.
The matrices of structure constants $A$ are calculated in paper~\cite{classification}. We give a reference to it in each case.

\emph{Case}~\ref{1}~\cite[\S\,5.2.1]{classification}. The longest arcs exist by Lemma~\ref{lem-anh} and Corollary~\ref{crl-exist}. Indeed,
$$
A = \left(
\begin{array}{ccc}
0 & 0 & 0 \\
0 & 0 & 0 \\
0 & 0 & 1 \\
\end{array}
\right), \qquad
\Ker{A} = \sspan{\{(1 \ 0 \ 0)^T, \ (0 \ 1 \ 0)^T\}}, \qquad (1 \ 0 \ 0)^T \in [\g,\g]^0 \cap \interior{(C^+_{\id})^{\vee}}.
$$

\emph{Case}~\ref{2suppl}~\cite[\S\,5.2.1]{classification}. The matrix of structure constants has the form ($\kappa = (c_{12}^2)^2 - (c_{12}^1)^2 \neq 0$)
$$
A = \left(
\begin{array}{ccc}
0 & 0 & 0 \\
0 & 0 & 0 \\
c_{12}^1 & c_{12}^2 & 1 \\
\end{array}
\right), \qquad
\Ker{A} = \sspan{\{(-c_{12}^2 \ c_{12}^1 \ 0)^T, \ (1 \ 0 \ -c_{12}^1)^T\}},
$$
then $(1 \ 0 \ -c_{12}^1)^T \in [\g,\g]^0 \cap \interior{(C^+_{\id})^{\vee}}$.
Therefore, the longest arcs exist by Lemma~\ref{lem-anh} and Corollary~\ref{crl-exist}.

\emph{Cases}~\ref{3}--\ref{5}, \ref{7}~\cite[\S\,5.2.2]{classification}. A left-invariant time orientation 1-form does not exist by Lemma~\ref{lem-anh}. Indeed,
$$
A = \left(
\begin{array}{rrr}
1 & 1 & 0 \\
-1 & -1 & 0 \\
\tau & \tau & 1 \\
\end{array}
\right) \text{ or }
\left(
\begin{array}{rrr}
1 & -1 & 0 \\
1 & -1 & 0 \\
\tau & -\tau & 1 \\
\end{array}
\right), \quad
\Ker{A} = \sspan{\{(1 \ \mp1 \ 0)^T\}}, \quad [\g,\g]^0 \cap \interior{(C^+_{\id})^{\vee}} = \varnothing.
$$
Our sufficient condition for existence of the longest arcs is not applicable in these cases.

\emph{Cases}~\ref{11}--\ref{12}~\cite[\S\,5.2.3]{classification}. If $\chi = \kappa$ ($\chi \neq 0$), then
$$
A = \left(
\begin{array}{ccc}
0 & 2\chi & 0 \\
0 & 0 & 0 \\
0 & 0 & 1 \\
\end{array}
\right), \qquad
\Ker{A} = \sspan{\{(1 \ 0 \ 0)^T\}}, \qquad (1 \ 0 \ 0)^T \in [\g,\g]^0 \cap \interior{(C^+_{\id})^{\vee}}.
$$
Hence, the longest arcs exist by Lemma~\ref{lem-anh} and Corollary~\ref{crl-exist}.

If $\chi = -\kappa$ ($\chi \neq 0$), then
$$
A = \left(
\begin{array}{rrr}
0 & 0 & 0 \\
-2\chi & 0 & 0 \\
0 & 0 & 1 \\
\end{array}
\right), \qquad
\Ker{A} = \sspan{\{(0 \ 1 \ 0)^T\}}, \qquad [\g,\g]^0 \cap \interior{(C^+_{\id})^{\vee}} = \varnothing.
$$
Our sufficient condition for existence of the longest arcs is not applicable in this case.

\emph{Cases}~\ref{13}--\ref{15}~\cite[\S\,5.2.3]{classification}. The longest arcs exist by Lemma~\ref{lem-anh} and Corollary~\ref{crl-exist}. We have
$$
A = \left(
\begin{array}{ccc}
0 & 2\chi & 0 \\
0 & 0 & 0 \\
0 & c_{12}^2 & 1 \\
\end{array}
\right), \ (c_{12}^2)^2 = \kappa - \chi, \qquad
\Ker{A} = \sspan{\{(1 \ 0 \ 0)^T\}}, \qquad (1 \ 0 \ 0)^T \in [\g,\g]^0 \cap \interior{(C^+_{\id})^{\vee}}.
$$

\emph{Cases}~\ref{16}--\ref{18}~\cite[\S\,5.2.3]{classification}. A left-invariant time orientation 1-form does not exist by Lemma~\ref{lem-anh}. Indeed,
$$
A = \left(
\begin{array}{rrr}
0 & 0 & 0 \\
-2\chi & 0 & 0 \\
c_{12}^1 & 0 & 1 \\
\end{array}
\right), \ (c_{12}^1)^2 = -\kappa - \chi, \qquad
\Ker{A} = \sspan{\{(0 \ 1 \ 0)^T\}}, \qquad [\g,\g]^0 \cap \interior{(C^+_{\id})^{\vee}} = \varnothing.
$$
Thus, our sufficient condition is not applicable in this case.
\end{proof}.

\section{\label{sec-sl2}Sub-Lorentzian structures on the universal cover of the Lie group $\SL_2(\R)$}

Let us describe the structure of the universal cover $G = \widetilde{\SL}_2(\R)$. Consider the following model of the Lie group $\SL_2(\R)$:
$$
\SL_2(\R) \backsimeq \SU_{1,1} = \left\{
\left(
\begin{array}{cc}
z & w \\
\bar{w} & \bar{z} \\
\end{array}
\right) \, \Bigm| \, z,w \in \CC, \ |z|^2 - |w|^2 = 1
\right\}.
$$
Then the universal cover is constructed as follows:
$$
\pi : G = \widetilde{\SL}_2(\R) \backsimeq \R \times \CC \rightarrow \SU_{1.1}, \quad
\pi(c, w) = \left(
\begin{array}{cc}
e^{ic}\sqrt{1 + |w|^2} & w \\
\bar{w} & e^{-ic}\sqrt{1 + |w|^2} \\
\end{array}
\right), \quad c \in \R, \, w \in \CC.
$$
From this construction, one can easily obtain formulas for multiplication $(c_1, w_1) \cdot (c_2, w_2) = (c, w)$ in the group $G$:
\begin{equation}
\label{eq-mult}
\begin{gathered}
c = c_1 + c_2 + \arctg{\frac{\Image{(w_1\bar{w}_2e^{-i(c_1+c_2)})}}{\sqrt{1+|w_1|^2}\sqrt{1+|w_2|^2} + \Real{(w_1\bar{w}_2e^{-i(c_1+c_2)})}}},\\
w = w_2\sqrt{1+|w_1|^2}e^{ic_1} + w_1\sqrt{1+|w_2|^2}e^{-ic_2}.\\
\end{gathered}
\end{equation}
The tangent space at the identity $\id = (0,0)$ of the group $G$ is isomorphic to $\g = \R \times \CC$, and the Killing quadratic form is proportional to the expression $-\xi^2 + |\zeta|^2$,
where $\xi \in \R$, $\zeta \in \CC$.
Consider a circular cone at the group identity, depending on a parameter $\eta \geqslant 0$
$$
C^+_{\eta} = \{(\xi,\zeta) \in \g \, | \, -\xi^2 + (\eta + 1)|\zeta|^2 \leqslant 0, \ \xi > 0\}.
$$
Note that if $\eta > 0$, then the cone $C^+_{\eta}$ lies inside the cone of the Killing form $C^+_0$.

We need to know how the left shift of the cone $C^+_{\eta}$ looks like.

\begin{lemma}
\label{lem-admissible-velocities}
For the left shift of the cone $C^+_{\eta}$, we have
$$
L_{(c,w) *} C^+_{\eta} = \left\{
\left(
\xi + \frac{\Image{w\bar{\zeta}e^{-ic}}}{\sqrt{1+|w|^2}}, \
\zeta\sqrt{1+|w|^2}e^{ic} - iw\xi
\right) \in \R \times \CC \, \Bigm| \, (\xi, \zeta) \in C^+_{\eta} \right\}.
$$
Here $L_{(c,w)}$ denotes the left shift by the element $(c,w) \in G$, and $L_{(c,w) *}$ is its differential.
\end{lemma}

\begin{proof}
Let $(\xi, \zeta) \in C^+_{\eta}$. Consider a curve $(c_1(t), w_1(t)) \in G$, $t \in [0, 1]$ such that
$$
c_1(0) = 0, \qquad \dot{c}_1(0) = \xi, \qquad w_1(0) = 0, \qquad \dot{w}_1(0) = \zeta.
$$
It remains to compute the left shift of this curve using the multiplication rule~\eqref{eq-mult}, and then compute its differential at zero
$\frac{d}{dt}|_{t=0} (c, w) \cdot (c_1(t), w_1(t))$.

Indeed, let us introduce the following notation
$$
f(t) = \frac{\Image{\left(w\bar{w}_1(t)e^{-i(c + c_1(t))}\right)}}{\sqrt{1+|w|^2}\sqrt{1+|w_1(t)|^2} + \Real{\left(w\bar{w}_1(t)e^{-i(c+c_1(t))}\right)}}.
$$
Then since $c_1(0) = 0$ and $w_1(0) = 0$, we have
$$
\frac{d}{dt}\Bigm|_{t=0}f(t) = \frac{\Image{\left( w \dot{\bar{w}}_1(0) e^{-ic} \right)}}{\sqrt{1+|w|^2}}, \qquad
\frac{d}{dt}\Bigm|_{t=0}\arctg{f(t)} = \frac{1}{1 + f(0)^2} \cdot \frac{d}{dt}\Bigm|_{t=0}f(t) = \frac{d}{dt}\Bigm|_{t=0}f(t),
$$
since $f(0) = 0$. This gives the form of the first component of the vector $L_{(c,w) *}(\xi,\zeta)$.

We obtain the second component of this vector from the equality
$$
\frac{d}{dt}\Bigm|_{t=0} \left( w_1(t) \sqrt{1+|w|^2}e^{ic} + w\sqrt{1+|w_1(t)|^2}e^{-ic_1(t)} \right) =
$$
$$
= \dot{w}_1(0)\sqrt{1+|w|^2}e^{ic} + w\frac{\dot{w}_1(0)\bar{w}_1(0) + w_1(0)\dot{\bar{w}}_1(0)}{2\sqrt{1+|w_1(0)|^2}}e^{-ic_1(0)} +
w\sqrt{1+|w_1(0)|^2}e^{-ic_1(0)}(-i\dot{c}_1(0)),
$$
where $c_1(0) = 0$ and $w_1(0) = 0$.
\end{proof}

\begin{theorem}
\label{th-sl2}
Consider a left-invariant generalized sub-Lorentzian structure on the group $G = \widetilde{\SL}_2(\R)$ given by a cone $C^+ \subset \g$.
If the cone $C^+$ lies inside the cone of the Killing form $C^+_0$, i.e., $C^+ \setminus 0 \subset \interior{C^+_0}$,
then for any attainable point there exists the longest arc leading to it.
\end{theorem}

\begin{proof}
First note that there exists $\eta > 0$ such that $C^+ \subset C^+_{\eta} \subset C^+_0$,
By Remark~\ref{rem-small-big}, it is sufficient to prove the existence of the longest arcs for the structure generated by the cone $C^+_{\eta}$.
We choose a suitable time orientation 1-form $\tau$ to apply the sufficient condition for existence of the longest arcs (Theorem~\ref{th-exist}).
Recall that by Remark~\ref{rem-solv-semisimple} the 1-form cannot be left-invariant. Put $\tau = dc$. Obviously, this form is closed.
It remains to check conditions~(1)--(2) of Theorem~\ref{th-exist} with respect to the standard Euclidean structure $\sqrt{c^2 + |w|^2}$ on the space $\R \times \CC$.

It follows from Lemma~\ref{lem-admissible-velocities} that for $v = L_{(c,w) *} (\xi, \zeta)$ we have
$$
\tau_{(c,w)}(v) = \frac{\xi\sqrt{1+|w|^2} + \Image{w\bar{\zeta}e^{-ic}}}{\sqrt{1+|w|^2}}, \qquad \text{where} \qquad  \xi^2 \geqslant (\eta + 1)|\zeta|^2.
$$
This implies the bounds
$$
\frac{\xi\left(\sqrt{1+|w|^2} - \frac{|w|}{\sqrt{\eta+1}}\right)}{\sqrt{1+|w|^2}}
\leqslant \tau_{(c,w)}(v) \leqslant
\frac{\xi\left(\sqrt{1+|w|^2} + \frac{|w|}{\sqrt{\eta+1}}\right)}{\sqrt{1+|w|^2}}.
$$
Moreover, the lower of these bounds is greater than zero, since $\xi > 0$ and $\sqrt{1+|w|^2} - \frac{|w|}{\sqrt{\eta+1}} > 0$, because $\eta > 0$.
Thus, $\tau_{(c,w)}(v) > 0$ for any $(c,w)$ and $v \in L_{(c,w) *} C^+_{\eta}$.
Hence, condition~(1) of Theorem~\ref{th-exist} is satisfied.

To check the sublinear growth condition~(2) of the same theorem, it is sufficient to prove that
if $v \in L_{(c,w) *} C^+_{\eta}$, then $\frac{|v|}{\tau_{(c,w)}(v)} \leqslant A + B\sqrt{c^2 + |w|^2}$ for some constants $A, B > 0$.
Note that the expression $\frac{\sqrt{1+|w|^2} \pm \frac{|w|}{\sqrt{\eta+1}}}{\sqrt{1+|w|^2}}$ tends to $1 \pm \frac{1}{\sqrt{\eta+1}}$ while $|w| \rightarrow +\infty$.
These limits are positive, since $\eta > 0$.
Therefore, there exist constants $C_1, C_2 > 0$ such that  $C_1 < \frac{\sqrt{1+|w|^2} \pm \frac{|w|}{\sqrt{\eta+1}}}{\sqrt{1+|w|^2}} < C_2$.
Hence,
\begin{equation}
\label{eq-estimation}
\frac{|v|}{\tau_{(c,w)}(v)} < \frac{\sqrt{C_2^2\xi^2 + |\zeta\sqrt{1+|w|^2}e^{ic} - iw\xi|^2}}{C_1\xi} =
\frac{1}{C_1}\sqrt{C_2^2 + \left|\frac{\zeta}{\xi}\sqrt{1+|w|^2}e^{ic} - iw\right|^2}.
\end{equation}
Using the estimate
$$
\left|\frac{\zeta}{\xi}\sqrt{1+|w|^2}e^{ic} - iw\right| \leqslant \left|\frac{\zeta}{\xi}\sqrt{1+|w|^2}e^{ic}\right| + |iw| \leqslant
$$
$$
\leqslant \frac{\sqrt{1+|w|^2}}{\sqrt{\eta+1}} + |w| \leqslant \frac{1+|w|}{\sqrt{\eta+1}} + |w|,
$$
we get
$$
\frac{|v|}{\tau_{(c,w)}(v)} < \frac{1}{C_1}\sqrt{C_2^2 + \left(\frac{1+|w|}{\sqrt{\eta+1}} + |w|\right)^2} =
\frac{C_2}{C_1}\sqrt{1 + \frac{1}{C_2^2}\left(\frac{1+|w|}{\sqrt{\eta+1}} + |w|\right)^2} \leqslant
$$
$$
\leqslant \frac{C_2}{C_1}\left(1 + \frac{1}{C_2}\left(\frac{1+|w|}{\sqrt{\eta+1}} + |w|\right)\right) = A + B|w|,
$$
where
$$
A = \frac{1}{C_1} \left(\frac{1}{\sqrt{\eta+1}} + C_2 \right) > 0, \qquad B = \frac{1}{C_1} \left(\frac{1}{\sqrt{\eta+1}} + 1 \right) > 0.
$$
It remains to apply Theorem~\ref{th-exist}.
\end{proof}

\begin{proof}[Proof of assertion~\emph{(2)} of Theorems~\emph{\ref{th-exist3d}, \ref{th-exist-general}}]
Consider \emph{case}~\ref{10} from Table~\ref{tb-classification}.
The matrix of structure constants has the form~\cite[\S\,5.2.3]{classification}
$$
A = \left(
\begin{array}{ccc}
0 & \kappa+\chi & 0 \\
\kappa - \chi & 0 & 0 \\
0 & 0 & 1 \\
\end{array}
\right), \qquad \kappa^2 - \chi^2 > 0.
$$
It is easy to see that the Killing form $K(V,W) = \tr{(\ad{V} \cdot \ad{W})}$ in coordinates $x_1,x_2,x_3$ associated with the basis $X_1, X_2, X_3$ has the form
$$
K(x_1,x_2,x_3) = 2(\kappa + \chi) x_1^2 - 2(\kappa - \chi) x_2^2 + 2(\kappa^2 - \chi^2) x_3^2.
$$
Consider the segment $I = \{(x_1, x_2, x_3) \in \g \, | \, x_1 = 1, \, x_2 \in [-1,1], \, x_3 = 0\}$ (a cross-section of the cone $C^+$).
Clearly, the cone $C^+$ lies inside the cone of the Killing form if and only if $K|_I < 0$, i.e.,
$$
K|_I(x_2) = - 2(\kappa - \chi) x_2^2 + 2(\kappa + \chi) < 0 \qquad \text{for any} \qquad  x_2 \in [-1,1].
$$
Since in case~\ref{10} $\kappa \neq \chi$, the following two situations are possible.

If the leading coefficient $-2(\kappa - \chi) < 0$, then the inequality $K|_I < 0$ holds if and only if $K|_I(0) = 2(\kappa + \chi) < 0$.
Hence, $\chi < \kappa < -\chi$, but this contradicts the condition of case~\ref{10}, namely $|\kappa| > -\chi$.

If the leading coefficient $-2(\kappa - \chi) > 0$, then the inequality $K|_I < 0$ holds if and only if
$K|_I(\pm 1) = -2(\kappa - \chi) + 2(\kappa + \chi) = 4\chi < 0$.
Hence, $\kappa < \chi < 0$, and the cone $C^+$ lies inside the cone of the Killing form when this condition is satisfied.
It remains to apply Theorem~\ref{th-sl2}.

In the remaining cases of Lie algebras $\sll_2(\R)$ from Table~\ref{tb-classification}, the cone $C^+$ intersects the boundary of the cone of the Killing form.
So, the existence condition from Theorem~\ref{th-sl2} is not applicable.

Indeed, in \emph{case}~\ref{2} the matrix of structure constants and the Killing form have the form~\cite[\S\,5.2.1]{classification}
$$
A = \left(
\begin{array}{ccc}
0 & \kappa & 0 \\
\kappa & 0 & 0 \\
0 & 0 & 1 \\
\end{array}
\right), \qquad \kappa = \gamma \neq 0, \qquad
K(x_1,x_2,x_3) = 2\kappa x_1^2 - 2\kappa x_2^2 + 2 \kappa^2 x_3^2.
$$
Repeating the reasoning in the analysis of case~\ref{10} above, we obtain that the cone $C^+$ lies inside the cone of the Killing form if and only if
$$
K|_I < 0 \qquad \Leftrightarrow \qquad
\left\{
\begin{array}{l}
-2\kappa < 0, \\
K|_I (0) = 2\kappa < 0, \\
\end{array}
\right. \qquad \text{or} \qquad
\left\{
\begin{array}{l}
-2\kappa > 0, \\
K|_I(\pm 1) = -2\kappa + 2\kappa = 0 < 0, \\
\end{array}
\right.
$$
but this never holds.

For \emph{cases}~\ref{6}, \ref{8}, we have the following matrix of structure constants and the Killing form~\cite[\S\,5.2.2]{classification}
$$
A = \left(
\begin{array}{ccc}
1 & \kappa \mp 1 & 0 \\
\kappa \pm 1 & -1 & 0 \\
0 & 0 & 1 \\
\end{array}
\right), \qquad K(x_1,x_2,x_3) = 2(\kappa \mp 1)x_1^2 - 4x_1x_2 -2(\kappa \pm 1)x_2^2 + 2\kappa^2 x_3^2.
$$
Whence, the conditions $K|_I(-1) = 4 \mp 4 < 0$ and $K|_I(1) = -4 \mp 4 < 0$ obviously do not hold.

For \emph{case}~\ref{19}, we have~\cite[\S\,5.2.4]{classification}
$$
A = \left(
\begin{array}{ccc}
\chi & \kappa & 0 \\
\kappa & -\chi & 0 \\
0 & 0 & 1 \\
\end{array}
\right), \qquad \chi \neq 0, \qquad
K(x_1,x_2,x_3) = 2\kappa x_1^2 - 4\chi x_1x_2 - 2\kappa x_2^2 + 2(\kappa^2 + \chi^2) x_3^2.
$$
Thus, the conditions
$$
K|_I(-1)= -2\kappa + 4\chi + 2\kappa = 4\chi < 0, \qquad
K|_I(1) = -2\kappa - 4\chi + 2\kappa = -4\chi < 0
$$
cannot hold simultaneously.
\end{proof}

\medskip

\textbf{Acknowledgments.}
The author is grateful to the anonymous referee for pointing out inaccuracies in the first version of the paper.

\end{document}